\documentclass[letterpaper,12pt,fleqn]{article}

\usepackage{amssymb}
\usepackage{amsthm}
\usepackage{amsfonts}
\usepackage{mathtools}
\usepackage{amscd}
\usepackage[latin2]{inputenc}
\usepackage{amsmath}
\usepackage{url}
\usepackage{bbm}
\usepackage{enumitem}
\usepackage[top=2cm, bottom=2cm, left=2.6cm, right=2.6cm]{geometry}

\newtheorem*{theorem*}{Theorem}
\newtheorem*{lemma*}{Lemma}
\newtheorem*{condition*}{Condition}

\newtheorem{theorem}{Theorem}[section]
\newtheorem{lemma}[theorem]{Lemma}

\newtheorem{definition}[theorem]{Definition}

\def\del{\partial}

\def\ra{\rightarrow}
\def\eps{\epsilon}

\newcommand{\CC}{\mathbb{C}}

\def\o{\omega}

\newcommand{\myfrac}[2]{\genfrac{}{}{}{0}{#1}{#2}}
\newcommand{\bino}[2]{\genfrac{(}{)}{0pt}{0}{#1}{#2}}

\begin{document}
\title{Analytic Hilbert fields and hermitian Hilbert bundles}
\author{Dat Tran}
\date{2014}
\maketitle

\section{Introduction}
This paper was directly motivated by the paper \cite{L-S} of Lempert and Sz\H{o}ke. In \cite{L-S}, the authors defined the notion of a smooth Hilbert field. Hilbert field was introduced by Godement in \cite{Go}. It also appeared in \cite{Di} with a continuous structure attached to it. The smooth Hilbert field can be thought of as a generalization of a smooth hermitian Hilbert bundle with a hermitian connection. Smooth Hilbert fields has a curvature, just like a Hilbert bundle endowed with a connection on it. In \cite{L-S}, Lempert and Sz\H{o}ke defined what it means to be an analytic Hilbert field and proved that any analytic Hilbert field with curvature zero induces a Hilbert bundle. They also provided an example of a Hilbert field with curvature zero, yet does not induce any Hilbert bundle.  In this paper, we will give an example of an analytic Hilbert field which does not come from any Hilbert bundle. This example occured as a direct image of a bundle in the dissertation of the author, but in this paper, it will be presented in a more general form.

For the sake of completeness, we first provide the definitions of a Hilbert bundle, Hilbert field, smooth and analytic Hilbert field. Most of these definitions are also in \cite{L-S}.

\begin{definition}
A Banach manifold is a Hausdorf space $M$ with an open cover $\mathcal{U}$ such that for every $U \in \mathcal{U}$, we are given a Banach space $B_U$ and a homeomorphism $\varphi_U$ between $U$ and an open subset $V_U$ of $B_U$. These pairs $(U,\varphi_U)$ are called coordinate charts and it is required that for every two charts $(U,\varphi_U),(U',\varphi_{U'})$, the composition $\varphi_{U'} \circ \varphi_{U}^{-1}$ should be smooth wherever it is defined.
\end{definition}

\begin{definition}
Let $M$ and $N$ be two Banach manifolds. A map $f: M \ra N$ is smooth if for any pair of charts $(U,\varphi_U)$, $(V,\psi_V)$ on $M$, resp. $N$, the map $\psi_V\circ f \circ \varphi_U^{-1}$ is smooth where it is defined.
\end{definition}

\begin{definition}
A smooth Hilbert bundle is a smooth map $\pi:H \ra S$ of Banach manifolds, each fiber $\pi^{-1}s$ is endowed with the structure of a complex vector space; for each $s \in S$, there should exist an open neighborhood $U \subset S$, a complex Hilbert space $X_U$ and a smooth map $F_U: \pi^{-1}U \ra X_U$ whose restriction to each fiber $\pi^{-1}s$, $s \in U$ is linear and $\pi \times F: \pi^{-1}U \ra U \times X_U$ is diffeomorphic.
\end{definition}

We use $\Gamma^{\infty}(S,H)$ (or $\Gamma^{\infty}(H)$ when there is no ambiguity about the base)  to denote the set of smooth sections of the smooth Hilbert bundle $H\ra S$.

\begin{definition}
A smooth Hermitian metric on a smooth Hilbert bundle $\pi: H \ra S$ is a function $h: H \oplus H \ra \CC$ such that any local trivialization $F_U: \pi^{-1}U \ra X_U$ as in the previous definition can be chosen so that $h(u,v) = \langle F(u), F(v) \rangle$ where $u,v$ are in the same fiber and $\langle,\rangle$ is the inner product of $X_U$. We call the pair $(H,h)\ra S$ a hermitian Hilbert bundle.
\end{definition}

Let $S$ be a smooth manifold, we use $\textup{Vect}S$ to denote the set of smooth complex valued vector fields on $S$. If $(H,h)\ra S$ is a smooth hermitan Hilbert bundle and $\varphi,\psi \in \Gamma^{\infty}(H)$, then $h(\varphi,\psi) \in C^{\infty}(S)$.

\begin{definition}
A connection $\nabla$ on a Hilbert bundle $H \ra S$ associates with every $\xi \in \textup{Vect} S$ a linear map $\nabla_{\xi}: \Gamma^{\infty}(H) \ra \Gamma^{\infty}(H)$. It is required that for every local trivialization \newline $F_U: \pi^{-1}U \ra U \times X_U$, there be a smooth map $A: \CC \otimes TU \ra U \times \textup{End}X_U$, mapping fiber to fiber, linear on each fiber $\CC \otimes T_sU$ such that on $U$, 
\[
F_U(\nabla_{\xi}\varphi) = \xi F_U(\varphi) + A(\xi)F_U(\varphi)
\]
for every $\varphi \in \Gamma^{\infty}(S,H)$.

If $(H,h) \ra S$ is a hermitian Hilbert bundle and $\nabla$ is a connection on $H$, then $\nabla$ is a hermitian connection if for every $\xi \in \textup{Vect}S$, and for every smooth section $\varphi, \psi$ of $H$, we have $\xi h(\varphi,\psi) = h(\nabla_{\xi}\varphi,\psi) + h(\varphi,\nabla_{\bar{\xi}}\psi)$.
\end{definition} 

Now, we go over the definitions of Hibert fields.

\begin{definition}
A Hilbert field is a map between sets $p: H \ra S$ such that for all $s \in S$, there is a Hilbert space structure endowed on $H_s=p^{-1}s$.
\end{definition}

\begin{definition}[Smooth Hilbert field]\label{D:Hfieldsmooth}
Let $S$ be a smooth manifold, and $H \ra S$ a Hilbert field. A smooth structure on $H$ is given by specifying a set $\Gamma^{\infty}$ of sections of $H$, closed under adddition and under multiplication by elements of $C^{\infty}(S)$, and linear operators $\nabla_{\xi}:\Gamma^{\infty} \ra \Gamma^{\infty}$ for each $\xi \in \textup{Vect}S$ such that for $\xi,\eta \in \textup{Vect}S$, $f\in C^{\infty}(S)$, $\varphi,\psi \in \Gamma^{\infty}$, the following conditions are satisfied.
\begin{enumerate}
\item $\nabla_{\xi +\eta} \varphi = \nabla_{\xi}\varphi + \nabla_{\eta}\varphi$; $\nabla_{(f\xi)}\varphi=f\nabla_{\xi}\varphi$; $\nabla_{\xi}(f\varphi) = (\xi f)\varphi +f\nabla_{\xi}\varphi$.
\item $h(\varphi,\psi) \in C^{\infty}(S,\CC)$ and $\xi h(\varphi,\psi) = h(\nabla_{\xi}\varphi,\psi) + h(\varphi,\nabla_{\bar{\xi}}\psi)$.
\item The set $\{\varphi(s): \varphi \in \Gamma^{\infty} \}$ is dense in $H_s$ for all $s \in S$.
\end{enumerate}
\end{definition}

\begin{definition}
Let $H \ra S$ be a Hilbert field with a smooth structure $(\Gamma^{\infty},\nabla)$, and $(\tilde{H},\tilde{h})\ra S$ be a hermitian Hilbert bundle with a hermitian connection $\tilde{\nabla}$. If there is a fiber preserving map $F: H \ra \tilde{H}$ such that $F|H_s$ are isometric for all $s$, $F\varphi \in \Gamma^{\infty}(S,\tilde{H}) \quad \forall \varphi \in  \Gamma^{\infty}$, and $\tilde{\nabla}_{\xi}(F\varphi) = F(\nabla_{\xi}\varphi), \quad\forall \varphi \in  \Gamma^{\infty}, \forall \xi \in \textup{Vect}~S$, we then say that $H\ra S$ induces from the hermitian Hilbert bundle $(\tilde{H},\tilde{h}) \ra S$.
\end{definition}

Let $H \ra S$ be a Hilbert field with smooth structure $(\Gamma^{\infty}, \nabla)$. For every $\xi, \eta \in \textup{Vect} S$, $\varphi \in \Gamma^{\infty}$, we set
\[
R(\xi,\eta) \varphi := \nabla_{\xi}\nabla_{\eta}\varphi - \nabla_{\eta}\nabla_{\xi}\varphi - \nabla_{[\xi,\eta]}\varphi. 
\]
and call the operator $R$ the curvature of this smooth Hilbert field.
It was proved in \cite{L-S}, that $R(\xi,\eta)\varphi(s)$ only depends on $\xi(s),\eta(s)$, and $\varphi(s)$. 
Therefore, for each $\xi$, $\eta$ in $T_s S \otimes \CC$, $R(\xi,\eta)$ is a densely defined operator on $H_s$, which we also denote $R(\xi,\eta)$. 

If $(H,h)\ra S$ is a hermitian Hilbert bundle, then the curvature of the bundle is defined exactly as the definition of the curvature of smooth Hilbert field above. Yet, it is known that for every $\xi,\eta \in \textup{Vect}S$, the operator $R(\xi,\eta)$ is a bounded operator on each fiber.

\begin{definition}[Analytic section of smooth Hilbert field with analytic base and analytic Hilbert field]
Let $S$ be a finite dimensional analytic manifold and $H\ra S$ a smooth Hilbert field with smooth structure $(\Gamma^{\infty},\nabla)$.
\begin{enumerate}
\item A section $\varphi \in \Gamma^{\infty}$ is said to be analytic if for any compact subset $C$ of $S$ and any finite subset $\Xi$ of $\textup{Vect}~S$ such that every $\xi \in \Xi$ is analytic in a neighborhood of $C$, there is an $\eps >0$ such that
\begin{equation}\label{D:anal-inq}
	\textup{sup}~\frac{\eps^m}{m!}h(\nabla_{\xi_m}\ldots\nabla_{\xi_1}\varphi)^{1/2}(s) < \infty
\end{equation}
where the \textup{sup} is taken over $m=0,1,\ldots$; $\xi_j \in \Xi$; and $s \in C$.
\item The set of analytic sections of $S$ is denoted $\Gamma^{\o}$.
\item If the set $\{\varphi(s): \varphi \in \Gamma^{\o} \}$ is dense in $H_s$ for all $s\in S$, then this Hilbert field is said to be analytic.
\end{enumerate}
\end{definition}

The method of defining analytic sections like above can also be applied in defining real-analytic functions. For examples of this, please see \cite{KP}.

\section{Example of an anlytic Hilbert field not coming from any bundle}

Let $S = \CC$ and $H \ra S$ be a Hilbert field such that every fiber $H_s$ is infinite dimensional and seperable. Define $\varphi_j$, $j = 0,1,\ldots$ be sections of this Hilbert field such that for each $s$, $\{\varphi_j(s)\}$ form an orthonormal basis on $H_s$.
Let 
\[
\Gamma^{\infty} =\left \{\sum_{l=0}^{k}a_l(s)\varphi_l| k < \infty, a_l\in C^{\infty}(S) \right\}.
\]
Since the set of linear combinations of an orthonormal basis will be dense in any Hilbert space, the set $\{ \varphi(s) | \varphi \in \Gamma^{\infty} \}$ is dense in $H_s$.

Let $k$ be a complex-valued, real-analytic function on $S$.
Let $\xi = \myfrac{\del}{\del s}$ and define:
\[
\nabla_{\xi}\varphi_j = (j+1)k\varphi_j
\qquad
\textup{and}
\qquad
\nabla_{\bar{\xi}}\varphi_j = -(j+1)\bar{k}\varphi_j.
\]
Since $\{ \xi(s) \}$ span $T^{(1,0)}_s S$ for all $s$, we can extend by linearity and get $\nabla_{\eta}: \Gamma^{\infty} \ra \Gamma^{\infty}$ as required by the definition of smooth structure for all $\eta \in \textup{Vect}S$. So, $(\Gamma^{\infty}, \nabla)$ is a smooth structure of the Hilbert field $H \ra S$.

\begin{lemma}\label{L:HR1}
All the sections $\varphi_j$ are analytic and consequently, the Hilbert field $H \ra S$ constructed above is analytic. Moreover, if $g$ is a real-valued, analytic, non-harmonic function, and $k=\myfrac{\del g}{\del s}$ then this Hilbert field does not induce any hermitian Hilbert bundle.
\end{lemma}
\begin{proof}
The goal here is to show that $\varphi_j$ are analytic, i.e. for each fixed $j$, for every compact set $K \subset S$, there is a $\delta >0$ such that
\[
\textup{sup}\myfrac{\delta^m}{m!}h(\nabla_{\eta_1}\ldots\nabla_{\eta_m}\varphi_j)^{1/2}(s) < \infty,
\]
where the sup is taken over all $m=0,1,\ldots$, $\eta_1,\dots,\eta_m \in \left\{\myfrac{\del}{\del s},\myfrac{\del}{\del\bar{s}}\right\}$, and all $s\in K$. It was shown in Corollary 3.3.4 of \cite{L-S} that we only need to prove this for $\eta_1,\dots,\eta_m \in \left\{\myfrac{\del}{\del s},\myfrac{\del}{\del\bar{s}}\right\}$, instead of for $\eta_1,\dots,\eta_m$ in any possible predetermined finite subset of $\textup{Vect}^{\omega}~S$.

Let $f$ be a smooth function on $S$, for $\eta \in \left\{\myfrac{\del}{\del s},\myfrac{\del}{\del\bar{s}}\right\}$, we have
\[
\nabla_{\eta}(f\varphi_j) = (\eta f)\varphi_j + a(\eta)f\varphi_j
\]
where
\[
a\left(\myfrac{\del}{\del s}\right) = (j+1)k
\quad \textup{ and } \quad
a\left(\myfrac{\del}{\del\bar{s}}\right) = -(j+1)\bar{k}
\]

Fix $K\subset S$, a nonnegative integer $j$, and an analytic function $f$ defined in a neighborhood of $K$. 
Since $f,g$ are all analytic in a neighborhood of $K$, there is an $\epsilon \in (0,1)$ and an $M>1$ such that
\begin{equation}\label{E:analytic1}
\hbox{sup}\myfrac{\epsilon^m}{m!}|\eta_1\ldots \eta_m h(s)| < M
\end{equation}
where the sup is taken over all $m=0,1,\ldots$; $\eta_1,\dots,\eta_m \in \left\{\myfrac{\del}{\del s},\myfrac{\del}{\del\bar{s}}\right\}$; \newline $h \in \left\{f, (j+1)k, -(j+1)\bar{k} \right\}$; and all $s\in K$. (This is a result in \cite{KP})

Before the next step of the proof, we want to introduce some notations.
\begin{enumerate}
\item If $I \subset \{ 1,\dots,m\}$ for $m \geq 1$ and $I = \{ i_1, \ldots, i_l\}$ where $ l \leq m$ and $ i_1 > \ldots > i_l$, then $\eta_I = \eta_{i_1}\ldots\eta_{i_l}$. We also write $\eta_{\emptyset}=id$.
\item $a_{i} = a(\eta_{i})$.
\item A collection of $I_1,\ldots, I_l$ is a sub-splitting of $\{1,\ldots,m\}$ if $I_{\alpha}\cap I_{\beta} = \emptyset$ for $\alpha \neq \beta$ and $\bigcup_{\alpha}I_{\alpha}~\subseteq~\{1,\ldots, m\}$. $I_{\alpha}$ can be empty set.
\item A $k$-splitting of $m$ consists of a sub-splitting $I_1,\ldots,I_k$ and $i_1>\ldots>i_{k-1}$ of $\{1,\ldots,m\}$ such that every element of $I_{\alpha}$ is larger than $i_{\alpha}$ for $\alpha<k$ and $\bigcup_{\alpha} I_{\alpha} \cup \{i_1, \ldots, i_{k-1}\} = \{ 1,\ldots,m\}$. Every $m$ has $k$-splittings for $k=1,\ldots,m+1$. For $m=0$, the empty set is the unique 1-splitting of $m$.
\end{enumerate}

Let $m$ be a non-negative integer. For $k=1,\ldots,m+2$, a $k$-splitting of $(m+1)$ is of the following two types.
\begin{enumerate}
\item A $k$-splitting $(I_1,\ldots,I_k;i_1,\ldots,i_{k-1})$ is of type 1 if $\bigcup_{\alpha}I_\alpha~\subset~\{1,\ldots,m\}$.
This means $i_1 = m+1$ due to $i_1> \ldots > i_{k-1}$. Because any element of $I_{1}$ is larger than $i_{1}$, $I_1 = \emptyset$. Suppose $k=1$, then  $\bigcup_{\alpha} I_{\alpha} \cup \{i_1, \ldots, i_{k-1}\} = \emptyset \cup \emptyset \neq \{1, \ldots, m+1 \}$ and we have a contradiction. So, no $1$-splitting is of type 1. Hence, any $k$-splitting $(I_1,\ldots,I_k;i_1,\ldots,i_{k-1})$ of $m+1$ of type 1 coresponds to a $(k-1)$-splitting of $m$ which is $(I_2,\ldots,I_k;i_2,\ldots,i_{k-1})$. Vice versa, for any $(k-1)$-splitting of $m$, we have a unique type 1 $k$-splitting of $m+1$ by adjoining $I_1 = \emptyset$ and $i_1=m+1$ to it. So there is a $1-1$ correspondence between type 1 $k$-splitting of $m+1$ and $(k-1)$-splitting of $m$.
\item Type 2 $k$-splittings are those that are not type 1. It means that if $(I_1,\ldots,I_k;i_1,\ldots,i_{k-1})$ is type 2, then $m+1 \in \bigcup_{\alpha} I_{\alpha}$. And so, there must be an $\alpha$ such that $I_{\alpha}$ contains $m+1$. If we replace this $I_{\alpha}$ by $\tilde{I}_{\alpha}= I_{\alpha} \setminus \{m+1\}$ we get a $k$-splitting of $m$. Suppose $k=m+2$, then by counting, $\{i_1,\ldots,i_{k-1}\} = \{ 1,\ldots,m+1\}$ and so there cannot be any $I_{\alpha}$ that contain $m+1$. Hence, there is no type 2 $(m+2)$-splitting. From these facts, we can see that for each $k$ in $\{1,\ldots,m+1\}$ there is a $1-k$ correspondence between the $k$-splitting of $m$ and the $k$-splitting of $m+1$ of type 2 as follow: for each $k$-splitting $(I_1,\ldots,I_k;i_1,\ldots,i_{k-1})$ of $m$, we acquire $k$ different $k$-splittings of $m+1$ by inserting $m+1$ to the begining of each $I_{\alpha}$.
\end{enumerate}

Let $\eta_1,\dots,\eta_m \in \left\{\myfrac{\del}{\del s},\myfrac{\del}{\del\bar{s}}\right\}$. For any $f$ as above, we define $T^{m} \in C^{\infty}(S)$.
We define
\[
T^{m} = \sum\limits_{k=1}^{m+1} \sum (\eta_{I_1}a_{i_1})\ldots (\eta_{I_{k-1}}a_{i_{k-1}})\eta_{I_k}f
\]
where the inner sum is taken over all distinct $k$-splittings,
\[
S_{1}^{m} = \sum\limits_{k=2}^{m+1} \sum (\eta_{I_1}a_{i_1})\ldots (\eta_{I_{k-1}}a_{i_{k-1}})\eta_{I_k}f
\]
where the inner sum is taken over all distinct type 1 $k$-splittings, and
\[
S_{2}^{m} = \sum\limits_{k=1}^{m} \sum (\eta_{I_1}a_{i_1})\ldots (\eta_{I_{k-1}}a_{i_{k-1}})\eta_{I_k}f
\]
where the inner sum is taken over all distinct type 2 $k$-splittings. 
Clearly, 
\begin{equation}
T^{m} = S_{1}^{m} + S_{2}^{m}.
\end{equation}
We have
\begin{equation}
S_{1}^{m+1} = \sum\limits_{k=2}^{m+2} \sum (\eta_{I_1}a_{i_1})\ldots (\eta_{I_{k-1}}a_{i_{k-1}})\eta_{I_k}f
           = a_{m+1} T^{m}
\end{equation}
as explained when we described the first collection.
We also have
\begin{equation}
S_{2}^{m+1} = \sum\limits_{k=1}^{m+1} \sum (\eta_{I_1}a_{i_1})\ldots (\eta_{I_{k-1}}a_{i_{k-1}})\eta_{I_k}f
\end{equation}
where we can group terms in the inner sum into distinct groups of $k$-splitting of $m+1$, each group consists of elements correspond to the same $k$-splitting of $m$. For each $k$-splitting $(I_1,\ldots,I_k;i_1,\ldots,i_{k-1})$ of $m$, $\eta_{m+1}[(\eta_{I_1}a_{i_1})\ldots (\eta_{I_{k-1}}a_{i_{k-1}})\eta_{I_k}f]$ gives us the sum of all the $k$-splittings of $m+1$ correspondsis to it. So 
\begin{flalign}
S_{2}^{m+1} &= \sum\limits_{k=1}^{m+1} \sum (\eta_{I_1}a_{i_1})\ldots (\eta_{I_{k-1}}a_{i_{k-1}})\eta_{I_k}f \\
         &= \sum\limits_{k=1}^{m+1} \sum \eta_{m+1}[(\eta_{I_1}a_{i_1})\ldots (\eta_{I_{k-1}}a_{i_{k-1}})\eta_{I_k}f] \\
         &= \eta_{m+1}[\sum\limits_{k=1}^{m+1} \sum (\eta_{I_1}a_{i_1})\ldots (\eta_{I_{k-1}}a_{i_{k-1}})\eta_{I_k}f] \\
         &= \eta_{m+1} T^{m}
\end{flalign}
where the inner sums of (3.5) and (3.6) are taken over all possible $k$-splitting of $m$.

\textbf{Claim:} $\quad \nabla_{\eta_m}\ldots\nabla_{\eta_1}(f\varphi_j)= T^{m}\varphi_j$.

We prove the above claim by induction. 

For $m=1$, $T^1 = \eta_{1}f + a_{1}f$ and $\nabla_{\eta_{1}}\varphi_j = (\eta_{1}f)\varphi_j + (a_{1}f)\varphi_j$. Clearly, the claim is true. 

Suppose the claim is true for $m$. We have
\begin{flalign*}
\nabla_{\eta_{m+1}}\ldots\nabla_{\eta_1}(f\varphi_j) &= \nabla_{\eta_{m+1}}(\nabla_{\eta_m}\ldots\nabla_{\eta_1}(f\varphi_j))
                                                  = \nabla_{\eta_{m+1}}(T^{m} \varphi_j) \\
                                                 &= \eta_{m+1}(T^{m})\varphi_j + a_{m+1}T^{m} \varphi_j 
                                                  = S_{2}^{m+1} \varphi_j + S_{1}^{m+1} \varphi_j 
                                                  = T^{m+1} \varphi_j.
\end{flalign*}

Therefore, we have
\begin{equation}\label{E:nablafphij}
\nabla_{\eta_m}\ldots\nabla_{\eta_1}(f\varphi_j) =[\sum\limits_{k=1}^{m+1} \sum (\eta_{I_1}a_{i_1})\ldots (\eta_{I_{k-1}}a_{i_{k-1}})\eta_{I_k}f]\varphi_j,
\end{equation}
where the inner sum is taken over all of the $k$-splittings of $m$.

Fix a $k$ in $\{1,\ldots, m+1 \}$ and pick any set of positive integers $l_1, \ldots , l_k$ such that $l_1+\ldots+l_k=m+1$. We call a summand in (\ref{E:nablafphij}) with $|I_1|= l_1 -1,\ldots,|I_{k-1}|=l_{k-1}-1, |I_k|=l_k$ a term of type $(l_1,\ldots,l_k)$. From equation (\ref{E:analytic1}), each of the terms of type $(l_1,\ldots,l_k)$ is bounded above by $$\myfrac{M^k}{\eps^{m+1-k}}(l_1-1)!\ldots(l_k-1)!.$$
The number of terms of type $(l_1,\ldots,l_k)$ is at most $\bino{m+1}{l_1;\ldots;l_k}$. So the sum of all the terms of type $(l_1,\ldots,l_k)$ is bounded above by
$$\bino{m+1}{l_1;\ldots;l_k}\myfrac{M^k}{\eps^{m+1-k}}(l_1-1)!\ldots(l_k-1)!. $$
For each fixed $k$, the number of choices of $(l_1,\ldots,l_k)$ is $\bino{m}{k-1}$. And hence we can bound both sides of equation (\ref{E:nablafphij}) for any $s \in K$ as below
\begin{flalign*}
&h(\nabla_{\eta_1}\ldots\nabla_{\eta_m}\varphi_j)^{1/2}(s)\\
& =\sum\limits_{k=1}^{m+1} \bino{m}{k-1}\bino{m+1}{l_1;\ldots;l_k}\myfrac{M^k}{\eps^{m+1-k}}(l_1-1)!\ldots(l_k-1)! \\
& \leq \sum\limits_{k=1}^{m+1}\bino{m}{k-1}(m+1)!\myfrac{M^k}{\eps^{m+1-k}} \\
& \leq \sum\limits_{k=0}^{m}\bino{m}{k}(m+1)!\myfrac{M^{k+1}}{\eps^{m-k}} \\
& \leq \myfrac{(m+1)!M}{\eps^{m}}\sum\limits_{k=0}^{m}\bino{m}{k}(M\eps)^{k} \\
& \leq (m+1)!M(\myfrac{1+M\eps}{\eps})^m
\end{flalign*}

Pick $\delta = \myfrac{\eps}{2(1+M\eps)}$, we get
\begin{flalign*}
& \myfrac{\delta^m}{m!}h(\nabla_{\eta_1}\ldots\nabla_{\eta_m}\varphi_j)^{1/2}(s) \\
& \leq \left(\myfrac{\eps}{2(1+M\eps)}\right)^{m}\myfrac{1}{m!}(m+1)!M(\myfrac{1+M\eps}{\eps})^m \\
& \leq (m+1)M \left(\myfrac{1}{2}\right)^{m}   \longrightarrow 0 \textup{ as } m \ra \infty.
\end{flalign*}

So, we conclude 
\[
\hbox{sup}\myfrac{\delta^m}{m!}h(\nabla_{\eta_1}\ldots\nabla_{\eta_m}\varphi_j)^{1/2}(s) < \infty,
\]
where the sup is taken over all $m=0,1,\ldots$; $\eta_1,\dots,\eta_m \in \left\{\myfrac{\del}{\del s},\myfrac{\del}{\del\bar{s}}\right\}$; and all $s\in K$.

Taking $f\equiv 1$, this shows that $\varphi_j$ is an analytic section of $S$ for any fixed $j$. Since $\{\varphi_j(s)_{j=0}^{\infty}\}$ spans $H_s$ for any $s$, we have $H\ra S$ is an analythic Hilbert field.

Now, if $g$ is real-analytic and non-harmonic, and $k=\myfrac{\del g}{\del s}$, then the curvature operator $R\left(\myfrac{\del}{\del s},\myfrac{\del}{\del\bar{s}}\right)$ are unbounded operators at points where $\Delta g \neq 0$. This means that this Hilbert field does not induce any Hilbert bundle.

\end{proof}

\bibliographystyle{amsalpha}
\bibliography{paper1}

\providecommand{\bysame}{\leavevmode\hbox to3em{\hrulefill}\thinspace}
\providecommand{\MR}{\relax\ifhmode\unskip\space\fi MR }
\providecommand{\MRhref}[2]{%
  \href{http://www.ams.org/mathscinet-getitem?mr=#1}{#2}
}
\providecommand{\href}[2]{#2}
\begin{thebibliography}{God51}

\bibitem[Dix69]{Di}
Jacques Dixmier, \emph{Les {$C^{\ast} $}-alg\`ebres et leurs
  repr\'esentations}, Deuxi\`eme \'edition. Cahiers Scientifiques, Fasc. XXIX,
  Gauthier-Villars \'Editeur, Paris, 1969. \MR{0246136 (39 \#7442)}

\bibitem[God51]{Go}
R.~Godement, \emph{Sur la th\'eorie des repr\'esentations unitaires}, Ann. of
  Math. (2) \textbf{53} (1951), 68--124. \MR{0038571 (12,421d)}

\bibitem[KP02]{KP}
Steven~G. Krantz and Harold~R. Parks, \emph{A primer of real analytic
  functions}, second ed., Birkh\"auser Advanced Texts: Basler Lehrb\"ucher.
  [Birkh\"auser Advanced Texts: Basel Textbooks], Birkh\"auser Boston, Inc.,
  Boston, MA, 2002. \MR{1916029 (2003f:26045)}

\bibitem[LS14]{L-S}
L{\'a}szl{\'o} Lempert and R{\'o}bert Sz{\H{o}}ke, \emph{Direct {I}mages,
  {F}ields of {H}ilbert {S}paces, and {G}eometric {Q}uantization}, Comm. Math.
  Phys. \textbf{327} (2014), no.~1, 49--99. \MR{3177932}

\end{thebibliography}
\end{document}